\documentclass[11pt,reqno,a4paper]{amsart} 
\usepackage{graphicx}
\usepackage{amssymb,amsmath}
\usepackage{amsthm}
\usepackage{color,graphicx}
\usepackage{hyperref}
\usepackage{color}
\usepackage{mathabx}
\usepackage{subfigure}

\usepackage{subfigure} 

\usepackage{verbatim} 

\usepackage{relsize} 

\newcommand{\be}{\begin{equation}}

\newcommand{\ee}{\end{equation}}

\newcommand{\cn}{{\rm \,cn}}
\newcommand{\sn}{{\rm \,sn}}
\newcommand{\dn}{{\rm \,dn}}

\newcommand{\Ker}{{\rm \,Ker}}
\newcommand{\K}{{\rm \,K}}
\newcommand{\E}{{\rm \,E}}

\numberwithin{equation}{section}
\numberwithin{figure}{section}

\newtheorem{theorem}{Theorem}[section]
\newtheorem{proposition}[theorem]{Proposition}
\newtheorem{remark}[theorem]{Remark}
\newtheorem{lemma}[theorem]{Lemma}

\newtheorem{definition}[theorem]{Definition}

\begin{document}
\vglue-1cm \hskip1cm
\title[Cnoidal Waves for the quintic Klein-Gordon and Schrödinger Equations]{Cnoidal waves for the quintic Klein-Gordon and Schrödinger equations: Existence and Orbital Instability}

\begin{center}



\maketitle

{\bf Gabriel E. Bittencourt Moraes}

{Departamento de Matem\'atica - Universidade Estadual de Maring\'a\\
	Avenida Colombo, 5790, CEP 87020-900, Maring\'a, PR, Brazil.}\\
{ pg54546@uem.br}

{\bf Guilherme de Loreno}

{Departamento de Matem\'atica - Universidade Estadual de Maring\'a\\
	Avenida Colombo, 5790, CEP 87020-900, Maring\'a, PR, Brazil.}\\
{ pg54136@uem.br}

\vspace{3mm}

\end{center}

\begin{abstract}
\noindent In the present paper, we establish the existence and orbital instability results of \textit{cnoidal} periodic
waves for the quintic Klein-Gordon  and nonlinear Schrödinger equations. The spectral analysis for the corresponding
linearized operator is established by using the Floquet theory. The orbital instability is determined by applying an abstract result due to Shatah and Strauss.
\end{abstract}

\section{Introduction} 

This paper concerns new results about the orbital instability of periodic cnoidal standing waves for the quintic nonlinear Klein-Gordon equation (QKG)
\begin{equation}\label{KF2}
u_{tt}-u_{xx}+u-|u|^{4}u=0,
\end{equation} 
and the quintic nonlinear Schrödinger equation (QNLS)
\begin{equation}\label{NLS-equation}
	iu_t + u_{xx} + |u|^4u = 0.
\end{equation}
In both equations $u = u(x,t)$, $(x,t) \in \mathbb{R}{\times}\mathbb{R}_+$ is a complex-valued function and $L-$periodic at the first variable.

The nonlinear Klein-Gordon equation has several physics applications, for instance, particle physics, physical problems such as ferroelectric transitions,  crystal growths, dislocations, plasma physics, fluid mechanics and other related.  The nonlinear Schrödinger equation arises in various physical and biological contexts, for example, in nonlinear optics, for Bose-Einstein condensates, in the description of nonlinear waves such as propagation of a lase beam, water waves at the free surface of an ideal fluid and plasma waves. In the modelling of the DNA and it appears also in mesoscopic molecular structures.

The orbital stability for the KG has been extensively studied in the last decades. In \cite{shatah}, Shatah gave sufficient conditions for the orbital stability of the $n$-dimensional KG equation
\begin{equation}\label{KGShatah0}
	u_{tt}-\Delta u+u+f(|u|){\rm arg}u=0,
\end{equation} In \cite{grillakis}, Grillakis studied sufficient conditions for the orbital instability of standing waves of the form $u(x,t)=e^{ict}\varphi(x)$ related to the following equation
\begin{equation}\label{KGGrillakis}
	u_{tt}-\Delta u+u-f(|u|^2)u=0.
\end{equation}
This result has been generalized by Jeanjean and Le Coz in \cite{jeanjean} and they used the Mountain Pass Theorem to show the existence of minimizers for a certain  constrained functional.

Consider the Klein-Gordon equation with $p$-power nonlinearity posed in $\mathbb{R}^n$
\begin{equation}\label{KGShatah}
	u_{tt}-\Delta u +u-|u|^{p-1}u=0.
\end{equation}
For the case $p\in \mathbb{N}$ and $n\geq 3$, Shatah in \cite{shatah}  proved a result of orbital stability for the standing wave when $1<p<1+\frac{4}{n}$. An interesting result given in  \cite{shatahstrauss} by Shatah and Strauss proved that the standing wave solution is orbitally unstable when $p\geq 1+\frac{4}{n}$ and $|c|<1$. Wu in \cite{wu} proved the orbital instability for the standing wave when $n=1$, $1<p<5$, $c \in (-1,1)$ and $|c|=2^{-1}\sqrt{p-1}$. The main arguments used are a modulation argument combined with a virial identity according to the pioneer works \cite{lecoz} and \cite{martel}. Still in the case $n=1$ in the periodic context, Natali and Pastor in \cite{NP2008} studied the case $p=3$. It has been shown results about the orbital stability of standing waves of \textit{dnoidal} type and orbital instability of standing waves with \textit{cnoidal} type. The main tool to this end is the classical theory due to Grillakis \textit{et al.} in \cite{grillakis1} and \cite{grillakis2}. In \cite{cardoso2013}, Natali and Cardoso considered the case $p=5$ to show the orbital instability of dnoidal standing wave solutions in the Sobolev space $H_{per}^1 \times L^2_{per}$ restricted to the even periodic functions. The main tool was a computational approach based on \cite{neves} useful to decide the behaviour of the non-positive spectrum of the associated linearized operator. After that, the orbital instability was determined by using the arguments in \cite{grillakis1}. In \cite{AN1}, Angulo and Natali also showed the same result for the case $p = 5$ using a different approach to decide the quantity and multiplicity of non-positive eigenvalues of the associated linearized operator.\\
\indent The study of spectral stability of periodic waves has important contributions concerning the Klein-Gordon type equations in the real case, that is, for $u=u(x,t)\in\mathbb{R}$. In fact, Bronski \textit{et al.} in \cite{BronskiJohnsonKapitula} used the theory of quadratic pencils to prove the spectral stability associated to a general second order PDE (see also \cite{StanislavovaStefanovLinear} and \cite{StanislavovaStefanovSpectral} for related topics). To do so, they calculated the  Hamiltonian-Krein index $\mathcal{K}_{{\rm Ham}}$ and proving that $\mathcal{K}_{{\rm Ham}}=0$ implies the spectral stability (for more details concerning the definition of Hamiltonian-Krein index, see \cite{DeconinckKapitula}).  In the complex case, Demirkaya \textit{et al.} in \cite{DemirkayaHakkaevStanislavovaStefanov} obtained the spectral stability of periodic waves associated to the Klein-Gordon equation $(\ref{KGShatah})$ for $n=1$, $p=3$ and $p=5$. In both cases, explicit positive solutions of dnoidal type have been determined and they proved the spectral stability of these waves. In both cases $p=3$ and $p=5$, it is well known that a cnoidal profile is also a periodic solution but questions concerning the spectral/orbital stability never has been treated in the current literature so far. Our intention is to give a positive answer concerning the orbital stability for the case $p=5.$\\
\indent We first describe our paper for the case of QKG equation. It is well known that $(\ref{KF2})$ can be seen as abstract Hamiltonian  by considering 
\begin{equation}\label{J}
J=\begin{pmatrix}
0 & 0 & 0 & 1 \\ 
0 & 0 & -1 & 0 \\ 
0 & 1 & 0 & 0 \\ 
-1 & 0 & 0 & 0
\end{pmatrix},
\end{equation}
the equation $(\ref{KF2})$ can be reduced in a simple form as
\begin{equation}\label{hamiltonian-kleingordon}
	\frac{d}{dt}U(t)=J E'(U),
\end{equation} 
where $U=(u,u_t)=(u,v)$ and $E'$ indicates the Fr\'echet derivative of the conserved quantity \linebreak $E:H^1_{per}([0,L]) \times L^2_{per}([0,L]) \rightarrow \mathbb{R}$ given by 
\begin{equation}\label{E}
E(u,v)
=\dfrac{1}{2}\int_{0}^{L} \left(|u_x|^2+|v|^2+|u|^2-\frac{|u|^6}{3}\right) \;dx.
\end{equation}

\indent Moreover, \eqref{KF2} has another conserved quantity  $F:H^1_{per}([0,L]) \times L^2_{per}([0,L]) \rightarrow \mathbb{R}$ given by
\begin{equation}\label{F}
F(u,v)=\text{Im} \int_0^L \bar{u}u_t\; dx= \int_0^L \left( \text{Re}\,u \; \text{Im}\,u_t-\text{Im}\,u \; \text{Re}\,u_t\right) \; dx.
\end{equation}

An important mathematical aspect concerning equation \eqref{KF2} is the existence of periodic standing waves solutions of the form
\begin{equation}\label{PW}
u(x,t)=e^{ict}\varphi_c(x),
\end{equation}
where $c \in \mathbb{R}$ represents the wave frequency and $\varphi=\varphi_c: \mathbb{R}\longrightarrow \mathbb{R}$ is an $L-$periodic  smooth function. Substituting \eqref{PW} into \eqref{KF2}, we have that $\varphi$ satisfies the following second order ordinary differential equation
\begin{equation}\label{KF3}
-\varphi''+\omega\varphi-\varphi^{5}=0,
\end{equation}
where $\omega:=1-c^2>0$.

In the periodic context and for a fixed $L>2\pi$, one can find an explicit solution for the equation \eqref{KF3} which depends on the Jacobi elliptic function of \textit{cnoidal} type as
\begin{equation}\label{Sol2}
\varphi(x)=\frac{a\,{\rm cn}\left(bx,k \right)}{\sqrt{1-q\,{\rm sn}^2(bx,k)}},
\end{equation} 
where $k\in(0,1)$ is called modulus of the elliptic function and the quadruple $(a,b,c,q) \in \mathbb{R}^4$ depend smoothly on the parameter $k \in (0,1)$.

We specify the dependence of $c$ in terms of $k\in(0,1)$ for a fixed $L>2\pi$. In fact, for $c \in (-1,1)$, we see that $1-c^2=\omega \in (0,1]$. Parameter $\omega$ is then given by
\begin{equation}\label{wcn}
\omega = {\frac {-16 \left( \K\left( k \right)  \right)^{2}
 \left(\left( -{k}^{2}+1 \right) \sqrt{{k}^{4}-{k}^{2}+1}+{k}^{4}-{
k}^{2}+1 \right)}{\left({k}^{2}-1-\sqrt{{k}^{4}-{k}^{2}+1}
 \right){L}^{2}}}.
\end{equation}

In this way, we managed to construct a smooth curve of solutions
$$c \in I=(-1,1)  \longmapsto\varphi \in H_{per,e}^2,$$
all of them with the same period $L>2\pi$. Here, $H_{per,e}^2$ indicates the space $H_{per}^2$ constituted by even functions. 

In addition, we define for $c \in I$, the Lyapunov functional 
$$G(u,v)=E(u,v)-cF(u,v).$$ 
From \eqref{KF3}, we obtain $G'(\varphi,c\varphi,0,0)=0$, that is, $(\varphi,c\varphi,0,0)$ is a critical point of $G$. 
Thus, the second step is to study the quantity and multiplicity of the non-positive eigenvalues associated to the linearized operator
$\mathcal{L}:= G''(\varphi,c\varphi,0,0)$ given by 
\begin{equation}\label{matrixop}
\mathcal{L}=\begin{pmatrix}
\vec{\mathcal{L}_1} & 0 \\
0 & \vec{\mathcal{L}_2}
\end{pmatrix}, 
\end{equation}
where
\begin{equation}\label{matrixopdiagonal}
\vec{\mathcal{L}_1}=\begin{pmatrix}
-\partial_x^2+1-5\varphi^4 & -c \\ 
-c & 1
\end{pmatrix} 
\qquad \text{and} \qquad
\vec{\mathcal{L}_2}=\begin{pmatrix}
-\partial_x^2+1-\varphi^4 & c \\ 
c & 1
\end{pmatrix},
\end{equation}
restricted to the even periodic space. To do so, it is enough to study the non-positive spectrum of the following Hill operators
\begin{equation}\label{Hilloperators12}
\mathcal{L}_1= -\partial_x^2+\omega-5\varphi^4 \quad \text{and} \quad \mathcal{L}_2= -\partial_x^2+\omega-\varphi^4.
\end{equation}

According to the arguments contained in \cite{neves}, we are enabled to prove that the operator $\mathcal{L}_2$ has a unique negative eigenvalue which is simple. Concerning operator $\mathcal{L}_1$, we determine that the number of negative eigenvalues is two. To obtain the same scenario for the operators $\mathcal{L}_i$, $i=1,2$ restricted to the even periodic functions, we need to use some addition arguments of the general theory for differential operators with periodic potentials in \cite{brown}. Thus, denoting by $\mathcal{L}_e$ the restriction of $\mathcal{L}$ in $(\ref{matrixop})$ to the even periodic functions, we obtain that the kernel of $\mathcal{L}_e$ is one dimensional and the number of negative eigenvalues, indicated by $\text{n}(\mathcal{L}_{e})$, is also three. By considering $\text{p}(d'') \in \mathbb{N}$ as the number of positive eigenvalues associated to the second derivative of the function  $d: I \rightarrow \mathbb{R}$ given by $d(c)=G(\varphi,c\varphi,0,0)$
we obtain, according with \cite{grillakis1} that if $\text{n}(\mathcal{L}_{e})-\text{p}(d'')$ is an odd number, the periodic wave $\varphi$ is linearly unstable for perturbations restricted to the even periodic functions. To show the orbital (nonlinear) instability, we employ the main result in \cite{ShatahStraussbook} (see also \cite[Theorem 3.2]{NataliPastor2014}).

\indent Now, our focus concerns equation $(\ref{NLS-equation})$ and related topics. Le Coz in \cite{lecozNLS} has determined the existence and orbital stability of standing wave solutions of the form $u(x,t)=e^{i\omega t}\varphi(x)$ for the $n$-dimensional Schr\"odinger equation with power nonlinearity given by
\begin{equation}\label{LecozNLS1}
	iu_t+\Delta u+|u|^{p-1}u=0,
\end{equation}
when $1<p<1+\tfrac{4}{n}$ and $\omega>0$. He used  the concentration-compactness principle due to Cazenave and Lions \cite{cazenave} to show the existence of a ground state $\varphi$ which solves the equation
\begin{equation}\label{ellipti}-\Delta\varphi+\omega\varphi-\varphi^p=0.
\end{equation} For the orbital stability, he employed the classical method in \cite{grillakis1}. Weinstein in \cite{weinsteinNLS} considered the same equation as $(\ref{LecozNLS1})$ for $p=2\sigma+1$ and  $\varphi$ a ground state solution which solves $(\ref{ellipti})$. He proved the orbital stability in $H^1(\mathbb{R}^n)$ for the case $\sigma<\tfrac{2}{n}$ with $n=1$ and $n=3$.

Concerning equation $(\ref{LecozNLS1})$ for the case $n=1$ in the periodic context, Angulo in \cite{angulo} established for the case $p=3$, the orbital stability of periodic standing waves solutions $\varphi$ with \textit{dnoidal} profile of the form $\varphi(x)=a\,{\rm dn}(bx,k)$ by combining the ideas  in \cite{bona}  and \cite{weinsteinNLS}. Moreover, it has been described that it is not possible to decide about the orbital stability of periodic waves with \textit{cnoidal} profile as $\varphi(x)=a\,{\rm cn}(bx,k)$ using the arguments in  \cite{grillakis1}. In a convenient interval $(0,\omega_1$), Gustafson \textit{et al.} in \cite{gustafson} established spectral stability results for the \textit{cnoidal} waves with respect to perturbations with the same period  $L$ and orbital stability results in the space  constituted by anti-periodic functions with period $L/2$. The orbital stability of periodic \textit{cnoidal} waves was determined by Natali \textit{et al.} in \cite{NMLP2021} in the same interval $(0,\omega_1)$ as above. However, the authors have restricted the analysis over the Sobolev space $H_{per}^1$ constituted by zero mean periodic functions. For the case $p=5$, Angulo and Natali in \cite{AN2} showed the existence of a unique $\omega^{*}>\frac{\pi^2}{L^2}$ such that the periodic wave with dnoidal profile as
\begin{equation}\label{sol21}
	\varphi(x)= \frac{\alpha \,{\rm dn}\left(\varrho x,k \right)}{\sqrt{1-\eta\,{\rm sn}^2(\varrho x,k)}}
\end{equation}
is orbitally stable for all $\omega\in\left(\frac{\pi^2}{L^2},\omega^{*}\right)$ and orbitally unstable for all $\omega\in (\omega^{*},+\infty)$.

\indent Concerning the spectral stability of periodic standing waves for the equation $(\ref{NLS-equation})$, we can cite  \cite{HakkaevStanislavovaStefanov}. In this case, Hakkaev \textit{et al.} established spectral stability results of positive standing waves for the equation $(\ref{NLS-equation})$ with dnoidal profile. Using numerical tools, the authors determined the existence of $\omega^* > 0$ such that for  $\omega \in (0, \omega^*)$ the wave is spectrally stable and for $\omega \in (\omega^*, \infty)$ is spectrally unstable.


\indent We shall give a brief explanation of our results concerning equation $(\ref{NLS-equation})$. Notice that \eqref{NLS-equation} can be also seen as an abstract Hamiltonian equation by considering $U = (\text{Re}\,u, \text{Im}\,u)$ as a solution of
\begin{equation}\label{hamiltonian-NLS}
	\frac{d}{dt} U(t) = \mathcal{J}\mathcal{E}'(U(t)),
\end{equation}
where
\begin{equation}\label{J-NLS}
	\mathcal{J}=\begin{pmatrix}
		0 & 1 \\ 
		-1 & 0  \\
	\end{pmatrix},
\end{equation}
and $\mathcal{E}'$ represents the Fréchet derivate of the conserved quantity $\mathcal{E} : H^1_{per}([0,L]) \rightarrow \mathbb{R}$ given by 
\begin{equation}\label{E-NLS}
	\mathcal{E}(U):= \mathcal{E}(u) = \frac{1}{2} \int_{0}^{L} \left( |u_x|^2 - \frac{1}{3} |u|^6 \right) dx.
\end{equation}
In addition, equation \eqref{NLS-equation} has another conserved quantity $\mathcal{F}: H^1_{per}([0,L]) \rightarrow \mathbb{R}$ expressed by
\begin{equation}\label{F-NLS}
	\mathcal{F}(u) = \frac{1}{2} \int_{0}^{L} |u|^2 dx.
\end{equation}

Equation (\ref{NLS-equation}) admits periodic standing wave solutions of the form
\begin{equation}\label{NLS-1}
	u(x,t) = e^{i\omega t} \varphi_\omega(x),
\end{equation}
where $\omega \in \mathbb{R}$ and $\varphi:= \varphi_\omega: \mathbb{R} \rightarrow \mathbb{R}$ is a smooth and $L$-periodic function. Substituting (\ref{NLS-1}) into (\ref{NLS-equation}), we obtain the following ordinary differential equation
\begin{equation}\label{NLS-2}
	-\varphi'' + \omega\varphi - \varphi^5 = 0,
\end{equation}
which is similar to (\ref{KF3}). As we have already seen before, for $\omega>0$, equation (\ref{NLS-2}) has solutions depending on the Jacobi elliptic function of cnoidal type as in \eqref{Sol2}. The wave frequency $\omega > 0$ is given explicitly by the relation (\ref{wcn}).

For a fixed $L>0$, we are enabled to construct a smooth curve of even periodic waves which solves $(\ref{KF3})$ and it is given by 
$$\omega \in \mathcal{I}=\left(\frac{4\pi^2}{L^2},+\infty\right)   \longmapsto \varphi_{\omega} \in H_{per,e}^2.$$
In addition, we define for $\omega \in \mathcal{I}$, the Lyapunov functional 
$$\mathcal{G}(u)=\mathcal{E}(u)+\omega \mathcal{F}(u).$$ 
From \eqref{KF3}, we obtain $\mathcal{G}'(\varphi,0)=0$, that is, $(\varphi,0)$ is a critical point of $\mathcal{G}$. 
 Next,  we study the exact behaviour of the non-positive spectrum related to the linearized operator
$\mathcal{L}:= \mathcal{G}''(\varphi,0)$ given by 
\begin{equation}\label{matrixop2}
\mathcal{L}=\begin{pmatrix}
\mathcal{L}_1 & 0 \\
0 & \mathcal{L}_2
\end{pmatrix}, 
\end{equation}
where $\mathcal{L}_1$ and $\mathcal{L}_2$ are given by \eqref{Hilloperators12}. Considering $\mathsf{d}(\omega)=\mathcal{G}(\varphi,0)$ and denoting by $\mathcal{L}_e$  the restriction of $\mathcal{L}$ over the subspace of even functions, we are enabled to conclude that the kernel of $\mathcal{L}_e$ is one dimensional and ${\rm n}(\mathcal{L}_e)=3$. Thus, we conclude that ${\rm n}(\mathcal{L}_e)-{\rm p}(\mathsf{d}'')$ is an odd number and, as in the case of the QKG equation, we conclude by \cite{grillakis2}, the linear instability of the wave $\varphi$. The orbital (nonlinear) instability can be determined similarly using the main result of \cite{ShatahStraussbook}.

Our paper is organized as follows: in Section \ref{section2} we present some basic notations. In Section \ref{section3}, we show the existence of a smooth curve of periodic standing wave solutions of cnoidal type for the equations \eqref{KF2} and \eqref{NLS-equation}. A brief introduction concerning the classical Floquet theory and a spectral analysis for the operators $\mathcal{L}$ and $\mathcal{L}_e$ are established in Section \ref{section4}. Finally, the orbital instability of  periodic standing waves with cnoidal profile is shown in Sections \ref{section5} and \ref{section6}.

\section{Notation}\label{section2}

\indent For $s\geq0$ and $L>0$, the (real) Sobolev space
$H^s_{per}:=H^s_{per}([0,L])$
consists of all periodic distributions $f$ such that
$$
\|f\|^2_{H^s_{per}}:= L \sum_{k=-\infty}^{\infty}(1+k^2)^s|\hat{f}(k)|^2 <\infty
$$
where $\hat{f}$ is the periodic Fourier transform of $f$. The space $H^s_{per}$ is a  Hilbert space with the inner product denoted by $(\cdot, \cdot)_{H^s}$. When $s=0$, the space $H^s_{per}$ is isometrically isomorphic to the space  $L^2([0,L])$ and will be denoted by $L^2_{per}:=H^0_{per}$ (see, e.g., \cite{Iorio}). The norm and inner product in $L^2_{per}$ will be denoted by $\|\cdot \|_{L^2}$ and $(\cdot, \cdot)_{L^2}$. To avoid an overloading of notation, we omit the interval $[0, L]$ of the space $H^s_{per}([0,L])$ and we denote it simply by $H^s_{per}$.

For $s\geq0$, we denote
$
H^s_{per,e}:=\{ f \in H^s_{per} \; ; \; f \:\; \text{is an even function}\}.
$
Endowed with the norm and inner product in $H^s_{per}$. In addition, we denote the Sobolev space $\mathbb{H}_{per}^s$ concerning the complex function $f=f_1+if_2$ as
$$
\mathbb{H}^s_{per}:= H^s_{per} \times H^s_{per}, \quad \mathbb{H}^s_{per,e}:=H^s_{per,e} \times H^s_{per,e} \quad \text{and} \quad \mathbb{L}_{per}^2:= L^2_{per} \times L^2_{per}
$$
equipped with their usual norms and scalar products.


The symbols $\sn(\cdot, k), \dn(\cdot, k)$ and $\cn(\cdot, k)$ represent the Jacobi elliptic functions of \textit{snoidal}, \textit{dnoidal}, and \textit{cnoidal} type, respectively. For $k \in (0, 1)$, ${\rm F}(\phi, k)$ and $\E(\phi, k)$  denote the complete elliptic integrals of the first and second kind, respectively, and  we denote by $\K(k)={\rm F}\left(\frac{\pi}{2},k\right)$ and $\E(k)=\E\left(\frac{\pi}{2},k\right)$, (see \cite{byrd}).



\section{Existence of a Smooth Curve of Periodic Waves of Cnoidal Type}\label{section3}
Let $L>0$ be fixed. Our purpose in this section is to present the existence of
$L$-periodic solutions $\varphi: \mathbb{R} \longrightarrow \mathbb{R}$ associated to the the ordinary differential equation 
\begin{equation}\label{ode1}
- \varphi''+\omega\varphi-\varphi^{5}=0.
\end{equation}
Concerning the case of the equation $(\ref{KF2})$, we have $\omega=1-c^2 \in (0,1)$. When $(\ref{NLS-equation})$ is being considered, we only assume $\omega >0 $ without the dependence of $c$.

In \cite{AN2} (see also \cite{cardoso2013}) the authors put forwarded that \eqref{ode1} admits periodic solution wave with dnoidal profile as
$$ 
\phi(x)=\frac{\alpha \,{\rm dn}\left(\varrho x,k \right)}{\sqrt{1-\eta\,{\rm sn}^2(\varrho x,k)}},
$$
where $k\in(0,1)$ and $\alpha, \varrho, \eta$ depend smoothly on $k \in (0,1)$. Motivated by the work in \cite{AN2}, we can consider the ansatz
\begin{equation}
\varphi(x)=\frac{a\,{\rm cn}\left(\frac{4\K(k)}{L}x ,k \right)}{\sqrt{1-q\,{\rm sn}^2\left(\frac{4\K(k)}{L}x,k\right)}},
\label{cnsol}\end{equation} 
in \eqref{ode1} to obtain that $\varphi$ is a periodic cnoidal solution of \eqref{ode1}. Parameters  $a,q \in \mathbb{R}$ are given by
\begin{equation}\label{valuea}
a =\frac{2(\K(k)^2(k^2-1-\sqrt{(k^4-k^2+1})(-k^2-1-\sqrt{k^4-k^2+1})L^2)^{1/4}}{L}
\end{equation}
and
\begin{equation}\label{valued}
q =k^2-1-\sqrt{k^4-k^2+1}.
\end{equation}

Moreover, $\omega$ can be expressed as
\begin{eqnarray}\label{valuew}
\omega & = & \frac{L^2a^2-16q^2k^2\K(k)^2+16qk^2\K(k)^2}{L^2q^2} \nonumber \\
& = &  {\frac {-16 \left( \K\left( k \right)  \right)^{2}
 \left(\left( -{k}^{2}+1 \right) \sqrt{{k}^{4}-{k}^{2}+1}+{k}^{4}-{
k}^{2}+1 \right)}{\left({k}^{2}-1-\sqrt{{k}^{4}-{k}^{2}+1}
 \right){L}^{2}}}.
\end{eqnarray}
For a fixed $L>0$, it is clear that the parameters $a,q$ and $\omega$ in $(\ref{valuea})-(\ref{valuew})$ depend smoothly on the parameter $k \in (0,1)$. Since the dependence of $\omega$ in terms of $k$ is strictly monotonic, we see that $a$ and $q$ also depend smoothly on $\omega$. For each $c \in (-1,1)$, we can determine a smooth curve $c\in(-1,1)\rightarrow \varphi_c$ of periodic solutions for the equation \eqref{ode1}. However, it should be noticed that by \eqref{valuew} we have  $ \omega \in \scriptstyle \left( \frac{4\pi^2}{L^2},1\right)$ and since $\omega=1-c^2$,  we need to assume $L>2\pi$.  We can establish the following result.

\begin{theorem}\label{cnoidalcurve}
Let $L>2\pi$ be fixed. Equation \eqref{ode1} has $L$-periodic solutions with cnoidal profile given by $(\ref{cnsol})$
where the parameters $a \in \left(\frac{2\sqrt{\pi}}{\sqrt{L}}, +\infty\right)$, $q \in (-2,-1)$ and $\omega \in I:= \left(\tfrac{4\pi^2}{L^2},1\right)$ depend smoothly on $k \in (0,1)$ and they are given by \eqref{valuea}, \eqref{valued} and \eqref{valuew}, respectively. In addition, since $\omega=1-c^2$ one has that
$$
c \in I  \longmapsto \varphi=\varphi_{c} \in H^2_{per,e}([0,L])
$$
is a smooth curve of periodic solutions for \eqref{ode1}.
\end{theorem}
\indent Concerning the QNLS $(\ref{NLS-equation})$, we have a similar result.
\begin{theorem}\label{cnoidalcurveNLS}
Let $L>0$ be fixed. Equation \eqref{ode1} has $L$-periodic solutions with cnoidal profile given by $(\ref{cnsol})$
where the parameters $a \in \left(\frac{2\sqrt{\pi}}{\sqrt{L}}, +\infty\right)$, $q \in (-2,-1)$ and $\omega \in \mathcal{I}:=\left(\tfrac{4\pi^2}{L^2},+\infty\right)$ depend smoothly on $k \in (0,1)$ and they are given by \eqref{valuea}, \eqref{valued} and \eqref{valuew}, respectively. In addition, one has that
$$
\omega \in \mathcal{I} \longmapsto \varphi=\varphi_{\omega} \in H^2_{per,e}([0,L])
$$
is a smooth curve of periodic solutions for \eqref{ode1}.
\end{theorem}

\section{Spectral Analysis}\label{section4}

\subsection{Floquet Theory Framework} \label{Floquet}
Before presenting the spectral analysis concerning the operators in \eqref{matrixop} and \eqref{matrixop2}, we need to recall some basic facts about the Floquet theory (for further details see \cite{east} and \cite{magnus}).

Consider $\mathcal{P}=\mathcal{P}_s:H_{per}^2([0,L]) \subset L_{per}^2([0,L]) \longrightarrow L_{per}^2([0,L])$ the Hill operator, given by
\begin{equation}\label{operatorP}
\mathcal{P}=-\partial_x^2+g(s, \varphi (x)),
\end{equation}
where $g$ is a smooth potential depending on $\varphi$ and on a parameter $s \in \mathcal{V}\subset \mathbb{R}$.\\
\indent According to the Oscillation Theorem (see \cite[Theorem 2.1]{magnus}), the spectrum of $\mathcal{P}$  is formed by an unbounded sequence of
real eigenvalues $(\lambda_n)_{n \in \mathbb{N}} \subset \mathbb{R}$ so that
\[
\lambda_0 < \lambda_1 \leq \lambda_2 <\lambda_3 \leq \lambda_4 <
\cdots\; < \lambda_{2n-1} \leq \lambda_{2n}\; \cdots,
\]
where equality means that $\lambda_{2n-1} = \lambda_{2n}$  is a double eigenvalue. Moreover, the spectrum  is characterized by the number of zeros of the eigenfunctions as: if $p \in D(\mathcal{P})$ is an eigenfunction associated to either $\lambda_{2n-1}$ or $\lambda_{2n}$, then $p$  has exactly $2n$ zeros in the half-open interval $[0, L)$.

Let $p$ be a nontrivial periodic solution of the equation 
\begin{equation}\label{zeqL}
-f''+g(s, \varphi (x))f=0,
\end{equation}
where $f \in C_b^{\infty}([0,L]).$ Here $C_b^{\infty}([0,L])$  indicates the space constituted by bounded real-valued smooth functions. By the classical Floquet theory, we obtain the existence of a solution $y$ of \eqref{zeqL} which is linearly independent with $p$ such that $\{p,y\}$ is the fundamental basis of solutions for the  Hill equation \eqref{zeqL}. Moreover,  there exists  $\theta \in   \mathbb{R}$ (depending on $y$ and $p$) such that
\begin{equation}\label{theta0}
y(x+L)=y(x)+\theta p(x),\ \ \ \ \ \mbox{for all}\ x\in\mathbb{R}.
\end{equation}
Constant $\theta \in \mathbb{R}$ measures how function $y$ is periodic. To be more precise, $\theta=0$ if and only if $y$ is periodic. This criterion is very useful to establish if the kernel of $\mathcal{P}$ is $1$-dimensional by proving that $\theta\neq0$. Concerning this fact, we have the following result.

\begin{proposition}\label{theta}
If $p$ is an eigenfunction associated to the eigenvalue $\lambda_k \in \mathbb{R}$, for some $k \in \mathbb{N},$ and $\theta \in \mathbb{R}$ is the constant given by \eqref{theta0}, then $\lambda_k \in \mathbb{R}$ is simple if and only if $\theta \neq 0$. Moreover, if $p$ has $2n$ zeroes over $[0,L)$, then $\lambda_k=\lambda_{2n-1}$ if $\theta<0$, and $\lambda_k=\lambda_{2n}$ if $\theta>0$.
\end{proposition}
\begin{proof}
See \cite[Theorem 3.1]{neves}.
\end{proof}

We also need the concept of isoinertial family of self-adjoint operators.

\begin{definition}\label{defi12}
Let $s \in \mathcal{V}$. The inertial index of the operator $\mathcal{P}_s$ is a pair ${\rm in}(\mathcal{P}_s):=(n,z)\in \mathbb{N}^2$, where $n \in \mathbb{N}$ denotes the dimension of the negative subspace of  and $z \in \mathbb{N}$ denotes the dimension of  $\Ker(\mathcal{P}_s)$.
\end{definition}

\begin{definition}\label{defi1234}
The family of linear operators  $\{\mathcal{P}_{s}=-\partial_x^2+g(s, \varphi)\;  ; \; s \in \mathcal{V}\}$ is said to be isoinertial if
 ${\rm in}(\mathcal{P}_{s})$ is constant for any $s \in \mathcal{V}$.
\end{definition}

Next result determines the behavior of the non-positive spectrum of the linear operator  \linebreak $\mathcal{P}=\mathcal{P}_{s}$, for $s \in \mathcal{V}$, in \eqref{operatorP} just by knowing it for a fixed value $s_0 \in \mathcal{V}$.

\begin{theorem}\label{isonertialP}
Let $s \in \mathcal{V}$ and $\mathcal{P}_{s}= -\partial_x^2+g(s,\varphi)$ be the Hill Operator defined in \eqref{operatorP}. If $\lambda=0$ is an eigenvalue of  $\mathcal{P}_{s}$ and $g$ is continuously differentiable in all variables, then the family of operators $\{\mathcal{P}_{s}\;  ; \; s \in \mathcal{V}\}$ is isoinertial.
\end{theorem}
\begin{proof}
See \cite[Theorem 3.1]{natali1}. 
\end{proof}

\subsection{Spectral Analysis for the Klein-Gordon Equation.}

Let $L>2\pi$ be fixed. Consider $c \in (-1,1)$, $1-c^2=\omega \in (0,1)$ and $\varphi=\varphi_{c}$ as the solution of $(\ref{ode1})$ given by Theorem \ref{cnoidalcurve}. In this section,  we study the spectral properties of the matrix operators
$$
\vec{\mathcal{L}_1},\vec{\mathcal{L}_2}: H^2_{per}([0,L]) \times L^2_{per}([0,L]) \subset \mathbb{L}_{per}^2([0,L])\longrightarrow \mathbb{L}_{per}^2([0,L])
$$
defined by
\begin{equation}\label{matrixoperators}
\vec{\mathcal{L}_1}=\begin{pmatrix}
-\partial_x^2+1-5\varphi^4 & -c \\ 
-c & 1
\end{pmatrix} 
\qquad \text{and} \qquad
\vec{\mathcal{L}_2}=\begin{pmatrix}
-\partial_x^2+1-\varphi^4 & c \\ 
c & 1
\end{pmatrix}.
\end{equation}

Operators $\vec{\mathcal{L}_1}$ and $\vec{\mathcal{L}_2}$ are the real and imaginary parts of the full linear operator
\begin{equation*}
		\mathcal{L}: H^2_{per}{\times}L^2_{per}{\times}H^2_{per}{\times}L^2_{per} \subset \mathbb{L}^2_{per} \times \mathbb{L}^2_{per} \longrightarrow \mathbb{L}_{per}^2 \times \mathbb{L}^2_{per}
\end{equation*}
given by
\begin{equation}\label{fulloperator}
\mathcal{L}=\begin{pmatrix}
-\partial_x^2+1-5\varphi^4 & -c & 0 & 0 \\ 
-c & 1 & 0 & 0 \\ 
0 & 0 & -\partial_x^2+1-\varphi^4 & c \\ 
0 & 0 & c & 1
\end{pmatrix}.
\end{equation}

In order to obtain some spectral properties concerning $\vec{\mathcal{L}_1}$ and $\vec{\mathcal{L}_2}$, we need to present some preliminary results.
\begin{lemma}\label{lineoperator} 
Let $L>2\pi$ be fixed. Consider $c \in (-1,1)$, $1-c^2=\omega \in (0,1)$ and $\varphi=\varphi_{c}$ the  solution of $(\ref{ode1})$ given by Theorem \ref{cnoidalcurve}. Let us consider the self-adjoint operator,  $\mathcal{L}_1: H^2_{per}([0,L]) \subset L^2_{per}([0,L])  \longrightarrow L^2_{per}([0,L])$ given by
\begin{equation}\label{operatorL1}
\mathcal{L}_1:=-\partial_x^2+\omega-5\varphi^4.
\end{equation}
A number $\lambda\leq 0$ is an  eigenvalue of the operator $\vec{\mathcal{L}_1}$ given in \eqref{matrixoperators} if and only if the number
$$
\lambda\left(1-\frac{c^2}{\lambda-1}\right)\leq 0
$$
is an eigenvalue of the operator $\mathcal{L}_1$. A similar result can be determined if we compare operator $\vec{\mathcal{L}_2}$ given in \eqref{matrixoperators} with the self-adjoint operator  $\mathcal{L}_2: H^2_{per}([0,L]) \subset L^2_{per}([0,L])  \longrightarrow L^2_{per}([0,L])$ given by
\begin{equation}\label{operatorL2}
\mathcal{L}_2:=-\partial_x^2+\omega-\varphi^4.
\end{equation}
\end{lemma}
\begin{proof}
See \cite[Proposition 3.1]{cardoso2013}.
\end{proof}

\begin{remark}\label{negativenumber}
By Lemma \ref{lineoperator}, we have that the dimension of $\Ker(\mathcal{L}_i)$ is equal to the dimension of $\Ker(\vec{\mathcal{L}_i})$, $i=1,2$. 
\end{remark}

Consider the linear operator $\mathcal{L}_1$
given by \eqref{operatorL1}. We see that $\lambda=0$ is an eigenvalue with associated eigenfunction $p=\varphi'$. So, there exists a function $y$ which satisfies the Hill equation
\begin{equation}\label{HillL1}
-y''+\omega y-5\varphi^4 y=0,
\end{equation}
where $y \in C_b^{\infty}([0,L])$ and $\{\varphi',y\}$ is the fundamental set. Moreover, since $\varphi'$ is odd (due to $\varphi$ is even),  we have that $y$ is even. Thus, we see that $y$ satisfies the following initial value problem
\begin{equation}\label{PVI1}
\begin{cases}
-y''+\omega y-5\varphi^4 y=0 \\
y(0)=-\frac{1}{\varphi''(0)} \\
y'(0)=0
\end{cases}
\end{equation}
and the constant $\theta \in \mathbb{R}$ is given by
\begin{equation}\label{theta1}
\theta= \frac{y'(L)}{\varphi''(0)}.
\end{equation}

We fix $L=8$ and $k_0=0.5$. By \eqref{valuew}, we obtain  $1-c_0^2=\omega_0 \simeq 0.6403$ and the solution $\varphi$ can be determined by
\begin{equation}\label{cnaproximated}
\varphi(x)=\frac{1.2604\,\text{cn}(0.8428\,x, 0.5)}{\sqrt{1+1.6513\,\text{sn}^2(0.8428\,x,0.5)}}.
\end{equation}

Since, $\varphi''(0) \simeq -2.3742$, we can solve the initial value problem in \eqref{PVI1}. In fact, using a computational program, we see that $y(8) \simeq -32.7368$. Hence, we obtain that the constant $\theta \in \mathbb{R}$ given in \eqref{theta1} satisfies $\theta \simeq 13.7883>0$.  Proposition \ref{theta} and the fact of $\varphi'$ has two zeroes over the interval $[0,L)$ enable us to say that $\lambda=0$ is a simple eigenvalue and $\mathcal{L}_1=\mathcal{L}_{1,\omega_0}$ has two negative eigenvalues, that is, ${\rm in}(\mathcal{L}_{1,\omega_0})=(2,1)$ for the fixed value $\omega_0 \in (0,1)$. Thanks the Theorem \ref{isonertialP}, we obtain that the family of operators $\{\mathcal{L}_{1,\omega} \; ; \; \omega \in (0,1)\}$ is isoinertial and therefore, we obtain ${\rm in}(\mathcal{L}_{1,\omega})=(2,1)$ for all $\omega \in (0,1)$.

Next tables illustrates some (approximate) values of $\theta$ for different values of $L>2\pi$.

\begin{table}[!h]
\begin{tabular}{|c|c|}
\hline 
\multicolumn{2}{|c|}{$L=7$} \\ 
\hline 
$k$ & $\theta$ \\
\hline 
$0.01$ & $8.07$ \\
\hline
$0.1$ & $8.07$ \\
\hline 
$0.3$ & $8.07$ \\
\hline 
$0.5$ & $8.08$ \\
\hline 
$0.7$ & $8.17$ \\
\hline 
$0.9$ & $8.98$ \\
\hline 
$0.99$ & $17.31$ \\
\hline 
\end{tabular}
\hspace{0.5cm} 
\begin{tabular}{|c|c|}
	\hline 
	\multicolumn{2}{|c|}{$L=4\pi$} \\ 
	\hline 
	$k$ & $\theta$ \\
	\hline 
	$0.01$ & $83.77$ \\
	\hline
	$0.1$ & $83.77$ \\
	\hline 
	$0.3$ & $83.79$ \\
	\hline 
	$0.5$ & $83.94$ \\
	\hline 
	$0.7$ & $84.83$ \\
	\hline 
	$0.9$ & $93.24$ \\
	\hline 
	$0.99$ & $179.76$ \\
	\hline 
\end{tabular}
\hspace{0.5cm} 
\begin{tabular}{|c|c|}
	\hline 
	\multicolumn{2}{|c|}{$L=20$} \\ 
	\hline 
	$k$ & $\theta$ \\
	\hline 
	$0.01$ & $537.53$ \\
	\hline
	$0.1$ & $537.53$ \\
	\hline 
	$0.3$ & $537.64$ \\
	\hline 
	$0.5$ & $538.61$ \\
	\hline 
	$0.7$ & $544.29$ \\
	\hline 
	$0.9$ & $598.25$ \\
	\hline 
	$0.99$ & $1153.37$ \\
	\hline 
\end{tabular} 
\hspace{0.5cm} 
\begin{tabular}{|c|c|}
	\hline 
	\multicolumn{2}{|c|}{$L=100$} \\ 
	\hline 
	$k$ & $\theta$ \\
	\hline 
	$0.01$ & $335953.53$ \\
	\hline
	$0.1$ & $335954.32$ \\
	\hline
	$0.3$ & $336023.83$ \\
	\hline 
	$0.5$ & $336629.65$ \\
	\hline 
	$0.7$ & $340184.46$ \\
	\hline 
	$0.9$ & $373904.20$ \\
	\hline 
	$0.99$ & $720856.38$ \\
	\hline 
\end{tabular}  
\end{table}

Summarizing the above, we can establish the following result.
\begin{lemma}\label{negative2}
Let $L>2\pi$ be fixed and consider $c \in (-1,1)$. If $\varphi=\varphi_{c}$ is the solution with cnoidal profile given in Theorem \ref{cnoidalcurve}, then the operator $\mathcal{L}_1$ given in \eqref{operatorL1} has exactly two negative eigenvalues which are simple and the zero  is the third eigenvalue which is simple with eigenfunction $\varphi'$. Moreover, the remainder of the spectrum is constituted by a discrete set of eigenvalues.
\end{lemma}

Regarding the operator 
$
\mathcal{L}_2
$
given by \eqref{operatorL2}, we see that $\lambda=0$ is an eigenvalue with eigenfunction associated $p=\varphi$. Using similar as done for the operator $\mathcal{L}_1$, we can assure  the existence of an odd function $y \in C_b^{\infty}([0,L])$ such that
\begin{equation}\label{PVI2}
\begin{cases}
-y''+\omega y-\varphi^4 y=0 \\
y(0)=0 \\
y'(0)=\frac{1}{\varphi(0)}
\end{cases}
\end{equation}
and 
\begin{equation}\label{theta2}
\theta= \frac{y(L)}{\varphi(0)}.
\end{equation}

Fix $k_0=0.5$ and $L=8$. By \eqref{valuew}, we have a fixed value $1-c_0=:\omega_0 \simeq 0.6403$ and the solution $\varphi$ is given by \eqref{cnaproximated} with $\varphi(0)\simeq 1.260$.  On other hand, we can solve numerically the initial value problem in \eqref{PVI2} to obtain  $y(8) \simeq -13.1854$. Hence, we obtain that  $\theta \in \mathbb{R}$ given in \eqref{theta2} satisfies  $\theta \simeq -10.4602<0$. Thus, Proposition \ref{theta} and the fact of $\varphi$ has two zeros over $[0,L)$ enable us to say that $\lambda=0$ is a simple eigenvalue and $\mathcal{L}_2=\mathcal{L}_{2,\omega_0}$ has one negative eigenvalue, that is ${\rm in}(\mathcal{L}_{2,\omega_0})=(1,1)$ for the fixed value $\omega_0 \in (0,1)$. Thanks to Theorem \ref{isonertialP}, we see that the family of operators $\{\mathcal{L}_{2,\omega} \; ; \; \omega \in (0,1)\}$ is isoinertial. Therefore, for all $\omega \in (0,1)$, we infer ${\rm in}(\mathcal{L}_{2,\omega})=(1,1)$.

Next tables illustrate the values of $\theta $ for a different values of $L>2\pi$.
\begin{table}[h]
	\begin{tabular}{|c|c|}
		\hline 
		\multicolumn{2}{|c|}{$L=7$} \\ 
		\hline 
		$k$ & $\theta$ \\
		\hline 
		$0.01$ & $-7.80$ \\
		\hline
		$0.1$ & $-7.80$ \\
		\hline 
		$0.3$ & $-7.82$ \\
		\hline 
		$0.5$ & $-8.01$ \\
		\hline 
		$0.7$ & $-8.93$ \\
		\hline 
		$0.9$ & $-14.52$ \\
		\hline 
		$0.99$ & $-63.10$ \\
		\hline 
	\end{tabular}
	\hspace{0.5cm} 
	\begin{tabular}{|c|c|}
		\hline 
		\multicolumn{2}{|c|}{$L=4\pi$} \\ 
		\hline 
		$k$ & $\theta$ \\
		\hline 
		$0.01$ & $-25.13$ \\
		\hline
		$0.1$ & $-25.13$ \\
		\hline 
		$0.3$ & $-25.21$ \\
		\hline 
		$0.5$ & $-25.81$ \\
		\hline 
		$0.7$ & $-28.77$ \\
		\hline 
		$0.9$ & $-46.79$ \\
		\hline 
		$0.99$ & $-203.36$ \\
		\hline 
	\end{tabular}
	\hspace{0.5cm} 
	\begin{tabular}{|c|c|}
		\hline 
		\multicolumn{2}{|c|}{$L=20$} \\ 
		\hline 
		$k$ & $\theta$ \\
		\hline 
		$0.01$ & $-63.66$ \\
		\hline
		$0.1$ & $-63.66$ \\
		\hline 
		$0.3$ & $-63.85$ \\
		\hline 
		$0.5$ & $-65.38$ \\
		\hline 
		$0.7$ & $-72.88$ \\
		\hline 
		$0.9$ & $-118.53$ \\
		\hline 
		$0.99$ & $-515.12$ \\
		\hline 
	\end{tabular} 
	\hspace{0.5cm} 
	\begin{tabular}{|c|c|}
		\hline 
		\multicolumn{2}{|c|}{$L=100$} \\ 
		\hline 
		$k$ & $\theta$ \\
		\hline 
		$0.01$ & $-1591.55$ \\
		\hline
		$0.1$ & $-1591.60$ \\
		\hline
		$0.3$ & $-1596.19$ \\
		\hline 
		$0.5$ & $-1634.50$ \\
		\hline 
		$0.7$ & $-1821.97$ \\
		\hline 
		$0.9$ & $-2963.34$ \\
		\hline 
		$0.99$ & $-12878.20$ \\
		\hline 
	\end{tabular}  
\end{table}

Summarizing the above, we can establish the following result.
\begin{lemma}\label{negative1}
Let $L>2\pi$ be fixed and consider $c \in (-1,1)$. If $\varphi=\varphi_{c}$ is the solution with cnoidal profile given in Theorem \ref{cnoidalcurve}, then the operator $\mathcal{L}_2$ given in \eqref{operatorL2} has exactly one negative eigenvalue which is simple and zero is the second eigenvalue which is simple with eigenfunction $\varphi$. Moreover, the remainder of the spectrum is constituted by a discrete set of eigenvalues.
\end{lemma}

As consequence of  Lemma \ref{lineoperator}, Remark \ref{negativenumber} and the Lemmas \ref{negative2} and \ref{negative1}, we have:

\begin{theorem}\label{matrixeigenvalues}
Let $L>2\pi$ be fixed and consider $c \in (-1,1)$. If $\varphi=\varphi_{c}$ is the solution with cnoidal profile given in Theorem \ref{cnoidalcurve}, then the following spectral properties holds:

\noindent \textbf{\textit{i)}} the operator  $\vec{\mathcal{L}_1}$ given in \eqref{matrixoperators} has exactly two negative
eigenvalues which are simple and  zero is the third eigenvalue which is simple with eigenfunction $(\varphi', c \varphi').$ Moreover, the remainder of the spectrum is constituted by a discrete set of eigenvalues.

\noindent \textbf{\textit{ii)}}  the operator  $\vec{\mathcal{L}_2}$ given in \eqref{matrixoperators} has exactly one negative
eigenvalue which is simple and zero is the second eigenvalue which is simple with eigenfunction $(\varphi, -c \varphi).$ Moreover, the remainder of the spectrum is constituted by a discrete set of eigenvalues.
\end{theorem}

Our intention is to analyse the full linearized operator $\mathcal{L}$ defined in \eqref{fulloperator} by counting its number of negative eigenvalues and proving that $\Ker(\mathcal{L})=[(\varphi',c\varphi',0,0), (0,0,\varphi,-c\varphi)]$. First of all, we see that the operator $\mathcal{L}$ in \eqref{fulloperator} is obtained by considering the conserved quantities $E$ and $F$ defined in $(\ref{E})$ and $(\ref{F})$ respectively. By defining $G=E-cF$, one has
$$G'(\varphi,c\varphi,0,0)=E'(\varphi,c\varphi,0,0)-cF'(\varphi,c\varphi,0,0)=0,$$ 
that is, $(\varphi,i c \varphi)=(\varphi,c\varphi,0,0)$ is a critical point of $G$. In addition, we have 
$$\mathcal{L}\cong G''(\varphi,c\varphi,0,0),$$
where the symbol $\cong $ means that $\mathcal{L}$ is identified with $G''(\varphi,c\varphi,0,0)$ by a convenient Riesz isomorphism. By \eqref{ode1}, we see that
$$
(\varphi',c\varphi',0,0), (0,0,\varphi,-c\varphi) \in \Ker(\mathcal{L}),
$$
and thus, by Theorem \ref{matrixeigenvalues} we see that $\Ker(\mathcal{L})=[(\varphi',c\varphi',0,0), (0,0,\varphi,-c\varphi)]$. Moreover, the remainder of the spectrum is discrete and bounded away from zero. Next, since $\mathcal{L}$ is a diagonal operator, we obtain
\begin{equation}\label{eigenvaluefulloperator}
\text{n}(\mathcal{L})=\text{n}(\vec{\mathcal{L}_1})+\text{n}(\vec{\mathcal{L}_2})=2+1=3.
\end{equation}

We turn our attention to the operator $\mathcal{L}$  restricted to the space of even functions. More precisely
\begin{equation}\label{fulloperatoreven}
	\mathcal{L}_e:= \mathcal{L} \,: \, (H^2_{per,e}{\times}L^2_{per,e}){\times}(H^2_{per,e}{\times}L^2_{per,e}) \subset \mathbb{L}^2_{per,e} \times \mathbb{L}^2_{per,e} \longrightarrow \mathbb{L}_{per,e}^2 \times \mathbb{L}^2_{per,e}.
\end{equation}

Concerning the operator $\mathcal{L}_{e}$, we have the following result.

\begin{theorem}\label{eveneigenfunctions}
Let $L>2\pi$ be fixed, consider $c \in (-1,1)$ and  $1-c^2=\omega \in (0,1)$. If  $\varphi=\varphi_{c}$ is the solution with cnoidal profile given in Theorem \ref{cnoidalcurve}, then the operator $\mathcal{L}_{e}$ defined in \eqref{fulloperatoreven} has exactly three simple negative eigenvalues and zero is a simple eigenvalue with eigenfunction $(0,0, \varphi,-c\varphi)$. Moreover the remainder of the spectrum is constituted by a discrete set of eigenvalues.
\end{theorem}
\begin{proof}
To prove this theorem, it is enough to establish that all eigenfunctions associated to the negative eigenvalues of $\mathcal{L}_1$ and $\mathcal{L}_2$, given respectively by \eqref{operatorL1} and \eqref{operatorL2},  are even functions. Indeed, if $g_i \in H^2_{per,e}([0,L])$ is an eigenfunction of $\mathcal{L}_i$, $i=1,2$ with eigenvalue $\gamma_i < 0$, then
$$
\vec{g_i}=\left(g_i, -\frac{c}{\mu_i-1}g_i \right) \in \mathbb{H}^2_{per,e}
$$
is an eigenfunction of the operator $\vec{\mathcal{L}_i}$ with eigenvalue $\mu_i <0$ which is the negative solution of $\gamma_i=\mu_i \left(1-\tfrac{c^2}{\mu_i-1}\right)$ (see Lemma \ref{lineoperator}).

By Lemma \ref{negative2}, let $\lambda_0<\lambda_1<0$ be the negative simple eigenvalues associated to $\mathcal{L}_1$ whose associated eigenfunctions are $\phi_0, \phi_1 \in H^2_{per}$, respectively. Moreover, $\lambda_2:=0$ is the third eigenvalue with eigenfunction $\varphi'$.  By Lemma \ref{negative1}, we also can take $\zeta_0<0$ the first negative eigenvalue of the operator $\mathcal{L}_2$ with associated eigenfunction $\chi_0 \in  H^2_{per}$. 

Using \cite[Theorem 1.1]{magnus}, we see clearly $\phi_0, \chi_0 \in H^2_{per,e}$ and we prove that $\phi_1 \in H^2_{per,e}$. Indeed, over $\left[0, \tfrac{L}{2}\right]$, let us consider the  eigenvalue problem 
\begin{equation}\label{problem2.E}
\begin{cases}
\mathcal{L}_1 f \equiv -f''+(\omega-5\varphi^4)f=\kappa f\\
f(0)=f\left(\tfrac{L}{2}\right)=0.
\end{cases}
\end{equation}

Since $\varphi'$ is periodic and odd, it follows that $\varphi'$ satisfies \eqref{problem2.E} for $\kappa=\lambda_2=0$. In addition, $\mathcal{L}_1\phi_1=\lambda_1 \phi_1$ with $\phi_1$ being either odd or even since $\phi_1$ solves a Hill equation with an even potential. Important to notice that the smallest eigenvalue of the problem \eqref{problem2.E} is associated with an odd eigenfunction (see \cite[Theorem 2.8.1]{brown}). Suppose that the smallest eigenvalue of \eqref{problem2.E} is $\lambda_1<0$, that is, $\phi_1$ is an odd function. Since $\lambda_2=0>\lambda_1$ and $\mathcal{L}_1\varphi'=0$, we obtain by \cite[Theorem 2.5.2]{brown} that $\varphi'$ has  infinitely many zeros, which is a contradiction and $\phi_1$ is even. The result is now proved.

\end{proof}

\subsection{Spectral Analysis for the Schrödinger Equation.}
Let $L>0$ be fixed. For $\omega>\frac{4\pi^2}{L^2}$, let $\varphi$ be the periodic solution of $(\ref{ode1})$ given by Theorem \ref{cnoidalcurveNLS}. In this subsection,  we study the spectral properties of the matrix operator
$$
\mathcal{L}: \mathbb{H}^2_{per}([0,L]) \subset \mathbb{L}^2_{per}([0,L]) \longrightarrow  \mathbb{L}^2_{per}([0,L]) 
$$
given by
\begin{equation}\label{fulloperatorNLS}
\mathcal{L}=\begin{pmatrix}
\mathcal{L}_1 & 0  \\ 
0 & \mathcal{L}_2 
\end{pmatrix}.
\end{equation}
where $\mathcal{L}_1$ and $\mathcal{L}_2$ are defined as in \eqref{operatorL1} and \eqref{operatorL2}, respectively. By Lemma \ref{negative2} and Lemma \ref{negative1}, we see that $\Ker(\mathcal{L})=[(\varphi',0),(0,\varphi)]$. Moreover, being $\mathcal{L}$ a diagonal operator, it follows that
$$
{\rm n}(\mathcal{L})={\rm n}(\mathcal{L}_1)+{\rm n}(\mathcal{L}_2)=2+1=3.
$$

In the sequel, we consider the operator
\begin{equation}\label{fulloperatorevenNLS}
\mathcal{L}_e:= \mathcal{L}:\mathbb{H}^2_{per,e}([0,L])  \subset \mathbb{L}^2_{per,e}([0,L]) \longrightarrow  \mathbb{L}^2_{per,e}([0,L]).
\end{equation}

From the proof of Theorem \ref{eveneigenfunctions}, we may conclude immediately the following result.

\begin{theorem}\label{eveneigenfunctionsNLS}
Let $L>0$ be fixed. Let $\omega>\frac{4\pi^2}{L^2}$ and $\varphi$ be the solution with cnoidal profile given by Theorem \ref{cnoidalcurveNLS}. The operator $\mathcal{L}_{e}$ defined in \eqref{fulloperatorevenNLS} has exactly three simple negative eigenvalues and zero is a simple eigenvalue with eigenfunction $(0,\varphi)$. Moreover the remainder of the spectrum is constituted by a discrete set of eigenvalues.
\end{theorem}

\section{Orbital instability of the cnoidal wave solutions for the QKG}\label{section5}

The goal of this section is to establish an orbital stability result. 
Consider $L> 2\pi$ and the complex energy space given by $X:= \mathbb{H}_{per}^1 \times \mathbb{L}_{per}^2$. It is well known that the equation \eqref{KF2} is invariant under  two basic symmetries: translation and rotation. In other words, if $u = u(x,t)$ is a solution of (\ref{KF2}), then $e^{-i\theta}u$ and $u(x+r,t)$ also solve the same equation for all  $\theta, r \in \mathbb{R}$.

In what follows, for $c \in (-1,1)$ and $\varphi$ the periodic solution given in Theorem \ref{cnoidalcurve}, we define
\begin{equation*}
\Phi:=(\varphi, ic \varphi)=(\varphi, c \varphi, 0, 0).
\end{equation*}
Our intention is to obtain a result of orbital instability  for the periodic solution $\varphi$ in $(\ref{cnsol})$ over the restricted subspace $X_e:=\mathbb{H}_{per,e}^1 \times \mathbb{L}_{per,e}^2\subset X$. To do so, it is necessary to consider the space $X$ only with the rotation symmetry since the translation is not an invariant over $X_e$.

Let $U= (u_1, v_1, u_2, v_2)$ and $\theta \in \mathbb{R}$. The rotation symmetry is given
$$
T_1(\theta) U= \begin{pmatrix}
\cos \theta & 0 & -\sin \theta & 0 \\ 
0 & \sin \theta & 0 & \cos \theta \\ 
\sin \theta & 0 & \cos \theta & 0 \\ 
0 & \cos \theta & 0 & -\sin \theta
\end{pmatrix} \begin{pmatrix}
u_1 \\ 
v_1 \\ 
u_2 \\ 
v_2
\end{pmatrix}.
$$

We can establish the following definition.

\begin{definition}\label{stadef}
We say that the periodic standing wave solution $\Phi$ is orbitally stable  if, for all $\varepsilon>0$, there exists $\delta>0$ such that, if $U_0 \in X_e$ satisfying
$$
\|U_0-\Phi\|_{X}<\delta,
$$
then the solution $U(t)$ of the \eqref{KF2}, with initial data $U(0)=U_0$, exists for all $t \in [0,t_0)$ for some $t_0>0$ and satisfies 
$$
\sup_{t\geq0}\inf_{\theta \in \mathbb{R}}\big\|U(t)-T_1(\theta)\Phi \big\|_{X}< \varepsilon.
$$
Otherwise, the standing wave solution is said  orbitally unstable.
\end{definition}


Before establishing our stability result, we present the following well-posedness result associated to the Klein-Gordon equation.

\begin{theorem}\label{wellposedness1}
The Cauchy problem associated with \eqref{KF2} is locally well-posed in $X_e$. More precisely, for each $U_0 \in X_e$ there exists a number $T>0$ and a unique solution  $U \in C([0,T], X_e)$ of equation \eqref{KF2} with $U(0)=U_0$ and for each  $T_0 \in (0,T)$ the mapping 
$$
U_0 \in X_e \longmapsto U \in C([0,T_0], X_e)
$$
is continuous. 
\end{theorem}
\begin{proof}
This result follows using the classical semigroup theory and density arguments. See \cite[Theorem 2.1]{cardoso2013}.
\end{proof}

\vspace{10pt}

\subsection{Convexity of the function $\boldsymbol{d}$.}\label{subsectionconvexityKG}
Let $L>2\pi$ be fixed. For $c\in (-1,1)$, let us consider the cnoidal wave $\varphi$ given by Theorem $\ref{cnoidalcurve}$. We see that $d: (-1,1) \subset \mathbb{R}\longrightarrow \mathbb{R} $ given by 
$$
d(c)=E(\varphi,c\varphi,0,0)-cF(\varphi,c\varphi,0,0),\; \text{for all} \; c \in (-1,1)
$$
is well defined, smooth in terms of $c$ and since $G'(\varphi,c\varphi,0,0)=0$, it follows that
$$
d'(c)=-F(\varphi,c\varphi,0,0)=-c\int_0^{L}(\varphi(x))^2\,dx.
$$
Thus, 
\begin{equation}\label{5.0}
d''(c)
=\, \displaystyle-\int_0^{L}(\varphi(x))^2dx+2c^2\frac{d}{d\omega}\int_0^{L}(\varphi(x))^2dx.
\end{equation}

One of the cornerstones for the study of the convexity of $d$ is to obtain a convenient expression for $\int_0^{L}(\varphi(x))^2dx.$ To do so, first we observe that
$$
\int_0^{L}(\varphi(x))^2dx=\int_0^{L} \frac{a^2\, \cn^2(bx,k)}{1-q\sn^2(bx,k)}\; dx=\frac{a^2L}{\K(k)}\int_0^{\K(k)}\frac{ \cn^2(u,k)}{1-q\sn^2(u,k)}\; du.
$$

\indent Using the explicit values in \eqref{valuea}, \eqref{valued} and \cite[formula 411.02]{byrd}, we obtain that
\begin{align}\label{Heumann}
\int_0^{L}(\varphi(x))^2dx  =& \, \frac{a^2L}{\K(k)}\frac{\pi(1-q)(1-\Lambda_0(\beta,k))}{2\sqrt{q(1-q)(q-k^2)}} \nonumber \\
 = & \, \frac{2\pi(1-q)\sqrt{\K(k)^2\,q\,(q-2k^2)\,L^2}(1-\Lambda_0(\beta,k))}{L\K(k)\,\sqrt{q(1-q)(q-k^2)}} \nonumber \\
=& \, \frac{-2\pi(\Lambda_0(\beta,k)-1)}{\sqrt{s(k)+1}}\sqrt{k^2+s(k)+1}\,\sqrt{2-k^2+s(k)}=:\mathrm{R}(k),
\end{align}
where
$$
s(k):=\sqrt{k^4-k^2+1},\; \text{for all} \; k \in (0,1),
$$
and $\Lambda_0$ indicates the Lambda Heumann function (see \cite[formula $150.03$]{byrd}) defined by 
$$
\Lambda_0(\beta,k):=\frac{2}{\pi}\left( \E(k)\,{\rm F}(\beta,k')+\K(k)\, {\rm E}(\beta,k')-\K(k)\, {\rm F}(\beta,k')\right),
$$
where
$$
\beta:= \arcsin\left(\frac{1}{\sqrt{1-q}}\right) \quad \text{and} \quad k':=\sqrt{1-k^2}.
$$

\indent By \eqref{5.0}, \eqref{Heumann} and the chain rule,  we have
\begin{equation}\label{dc2}
d''(c) =-\mathrm{R}(k)+2 \, c^2 \, \mathrm{R}'(k)\, \left(\frac{d\omega}{dk}\right)^{-1}.
\end{equation}

First, notice that for all $k \in (0,1)$, we have
\begin{equation}\label{positivityP}
-\mathrm{R}(k)=-\int_0^L (\varphi(x))^2\; dx=- \|\varphi\|_{L^2}^2<0.
\end{equation}

On the other hand, we can plot the graphic of the function $\mathrm{R}'$ (see Figure \ref{graphicdP}) to show the existence of a number $k^* \simeq 0.593$ such that
$
\mathrm{R}'(k)<0
$
for all $k\in(0,k^*)$.
\begin{figure}[h]
\centering
\includegraphics[width=6cm,height=5.5cm]{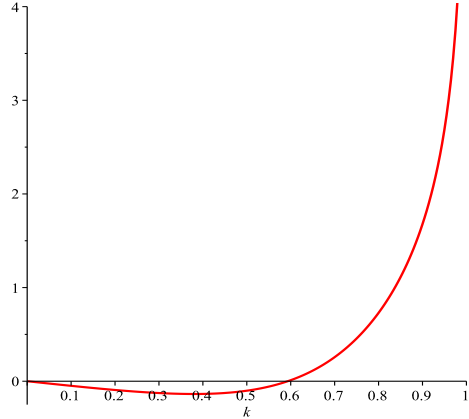}
\caption{Behavior of function $\mathrm{R}'$.} \label{graphicdP}
\end{figure}

Using the explicitly value of $\omega=1-c^2 \in (0,1)$ given in \eqref{valuew}, we have that $k^* \simeq 0.593$ defines uniquely a value $c^*$ given by
\begin{equation}\label{valuew*}
c^* \simeq \frac{0.65 \,\sqrt{-99.87+2.33L^2}}{L}>0.
\end{equation}

In addition, we have that $\omega \in (0,1)$ is a positive and strictly increasing function in terms of $k$, that is,
\begin{equation}\label{positivitydw}
\left(\frac{d\omega}{dk}\right)^{-1}>0,\ \ \ \mbox{for all}\ k\in(0,1).
\end{equation}

Thus, we have by \eqref{dc2}, the fact $R'(k)<0$ for all $k\in (0,k^*)$ and \eqref{positivitydw} that
\begin{equation}\label{d''} 
d''(c)<0,\; \text{for all} \; c \in(-1,-c^*]\cup[c^*,1).
\end{equation}

We have proved the convexity of $d$ over the interval $(-1,-c^*]\cup[c^*,1)$. Our next analysis is made by considering $k\in(k^*,1),$ that is, for $c\in(-c^*,c^*)$. In fact, for the case $ (k^*, 1),$ we do not have a precise answer for the convexity of the function $d$ since in this case this convexity depends on an eventual choice of the period $ L>2\pi $. Indeed, we turn back to the expression (\ref{5.0}) to have
\begin{equation}\label{d2inL}\begin{array}{llll}
	d''(c) & =& - \mathrm{R}(k) - 2\omega \, \mathrm{R}'(k) \left( \frac{d\omega}{dk} \right)^{-1} + \: 2 \, \mathrm{R}'(k) \left( \frac{d\omega}{dk} \right)^{-1}  \\\\
	& = &2 \, \mathrm{R}'(k) \, \mathrm{Q}(k) L^2 - \mathrm{R}(k) - 2\omega \, \mathrm{R}'(k) \left( \frac{d\omega}{dk} \right)^{-1}, 
	\end{array}\end{equation}
where $\mathrm{Q}:(0,1) \rightarrow \mathbb{R}_+^* $ is a function defined as
$
	\left( \frac{d\omega}{dk} \right)^{-1} = \: \mathrm{Q}(k) \, L^2.
$

Next, we define
\begin{equation}\label{alphaandbeta}
	\alpha=\alpha(k):= 2\, \mathrm{R}'(k) \, \mathrm{Q}(k) \quad \text{and} \quad  \beta=\beta(k):= \mathrm{R}(k) + 2\omega \, \mathrm{R}'(k) \left( \frac{d\omega}{dk} \right)^{-1}.
\end{equation}
We need to emphasize that $\alpha>0$ and $\beta>0$ only depend on the modulus $k$ since $\mathrm{R}$, $\mathrm{Q}$ and the product of $\omega$ with $\left( \frac{d\omega}{dk} \right)^{-1}$ only depend on $k$. By \eqref{d2inL} and \eqref{alphaandbeta}, we obtain for all $c\in(-c^*,c^*)$ that
\begin{equation}\label{d2second}
	d''(c) = \alpha(k) L^2 - \beta(k).
\end{equation}

Equation (\ref{d2second}) shows that  the positivity of $d''$ depends on an eventual choice of  $L$. Since \eqref{d2second} involves complicated functions to manipulate as the Lambda Heumann function, we can determine the convexity of $d$ using numerical calculations. In fact, we obtain the existence of $ L_0 \simeq 20.354 $ such that if $L \in (2\pi, L_0)$, we have 
$$
d''(c)  <0,\; \text{for all} \; c \in (-c^*,c^*).
$$
For the case $L>L_0$, we obtain a complicated scenario for the convexity of $d$. Indeed, according to the plots in the Figure \ref{figura2}, we can see that for some values of $L>L_0$, we obtain $d''(c)<0$ for some values of $k$, but when $L$ becomes large one has $d''(c)>0$. This last fact prevents us to use the abstract theory in \cite[Theorem 5.1]{grillakis2} for the linear instability since the difference $\text{n}(\mathcal{L}_e)-\text{p}(d'')=3-1=2$ is even.
\begin{figure}[!h]
	\subfigure[$L=20$]{
		\includegraphics[width=4.95cm,height=3.6cm]{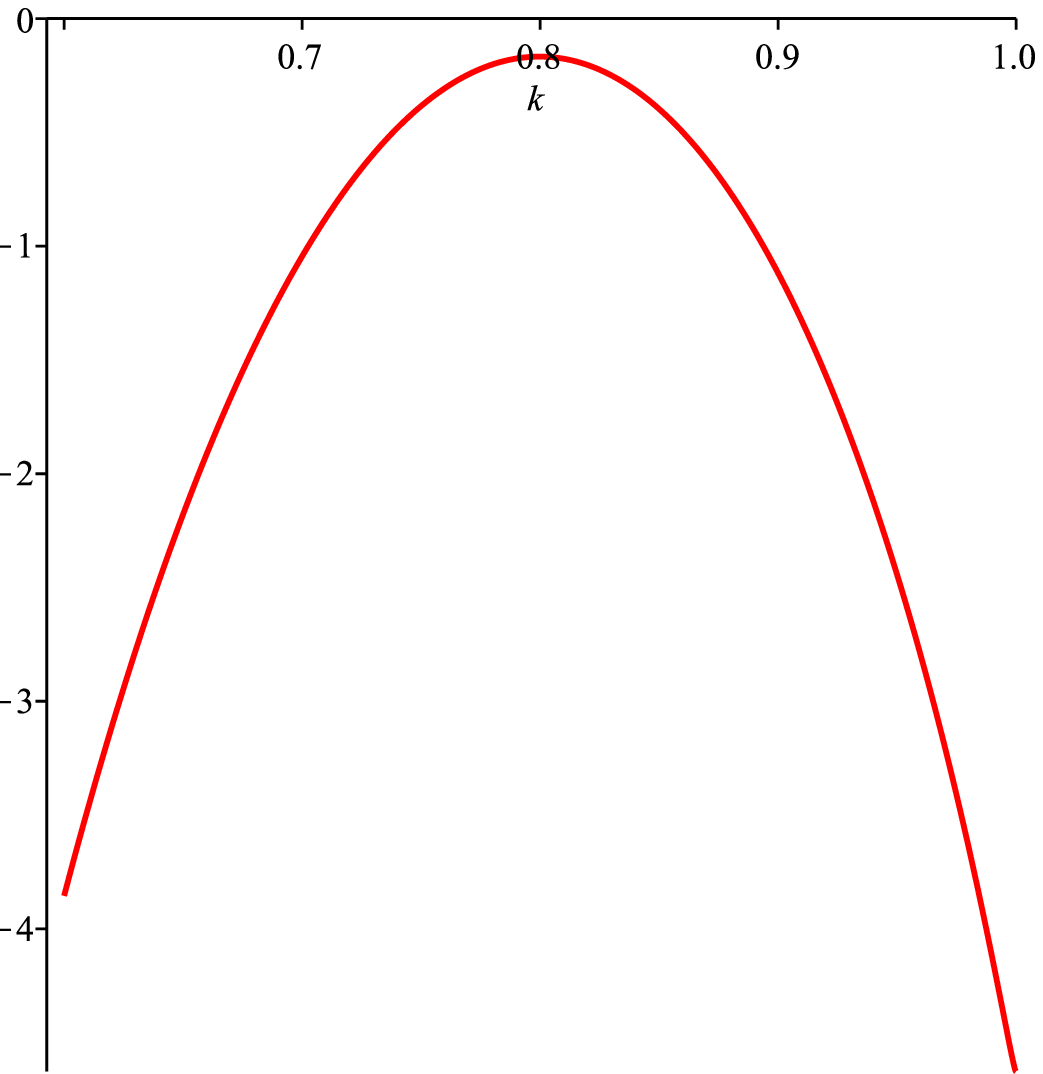}}
	\quad
	\subfigure[$L=20.2$]{
		\includegraphics[width=4.95cm,height=3.6cm]{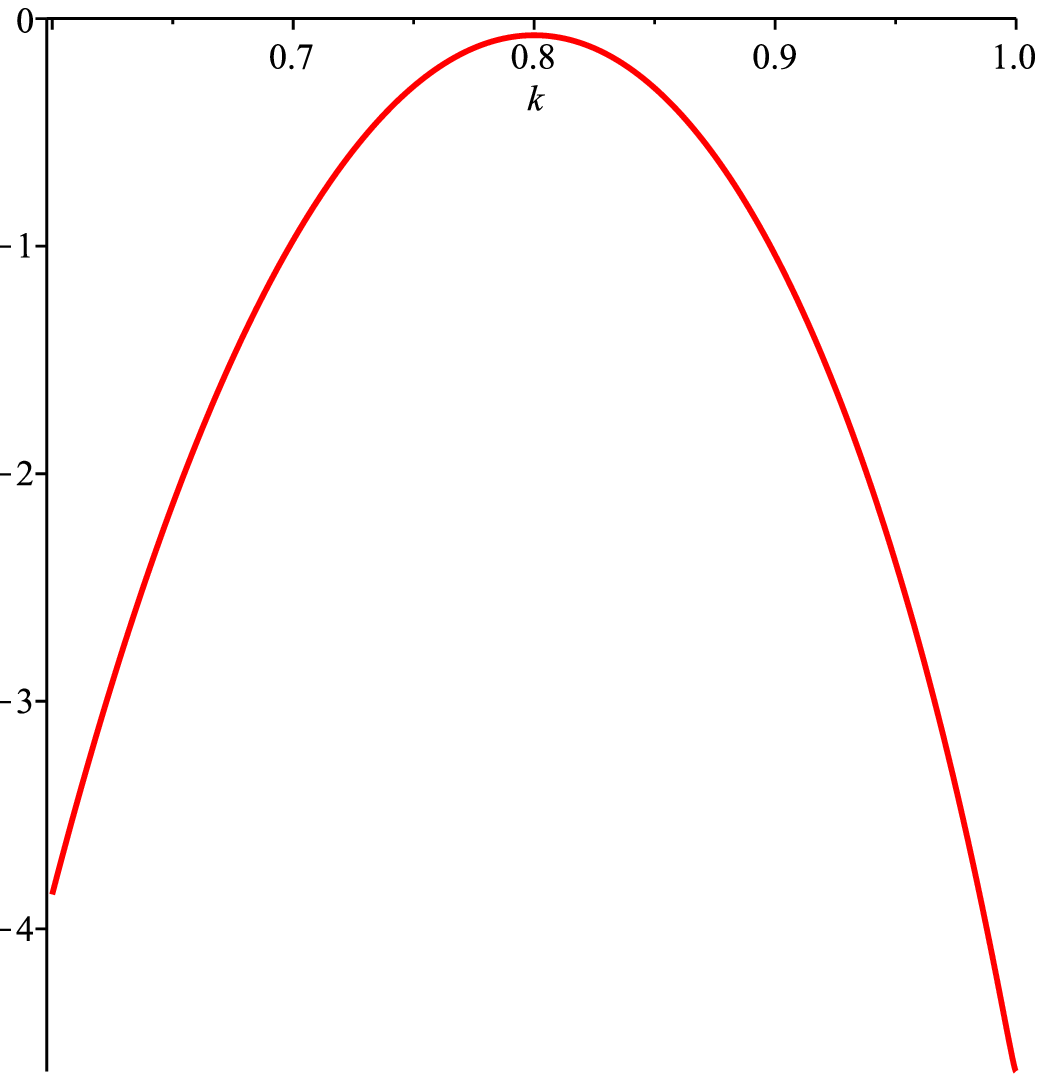}}
	\quad
	\subfigure[$L=20.4$]{
		\includegraphics[width=4.95cm,height=3.6cm]{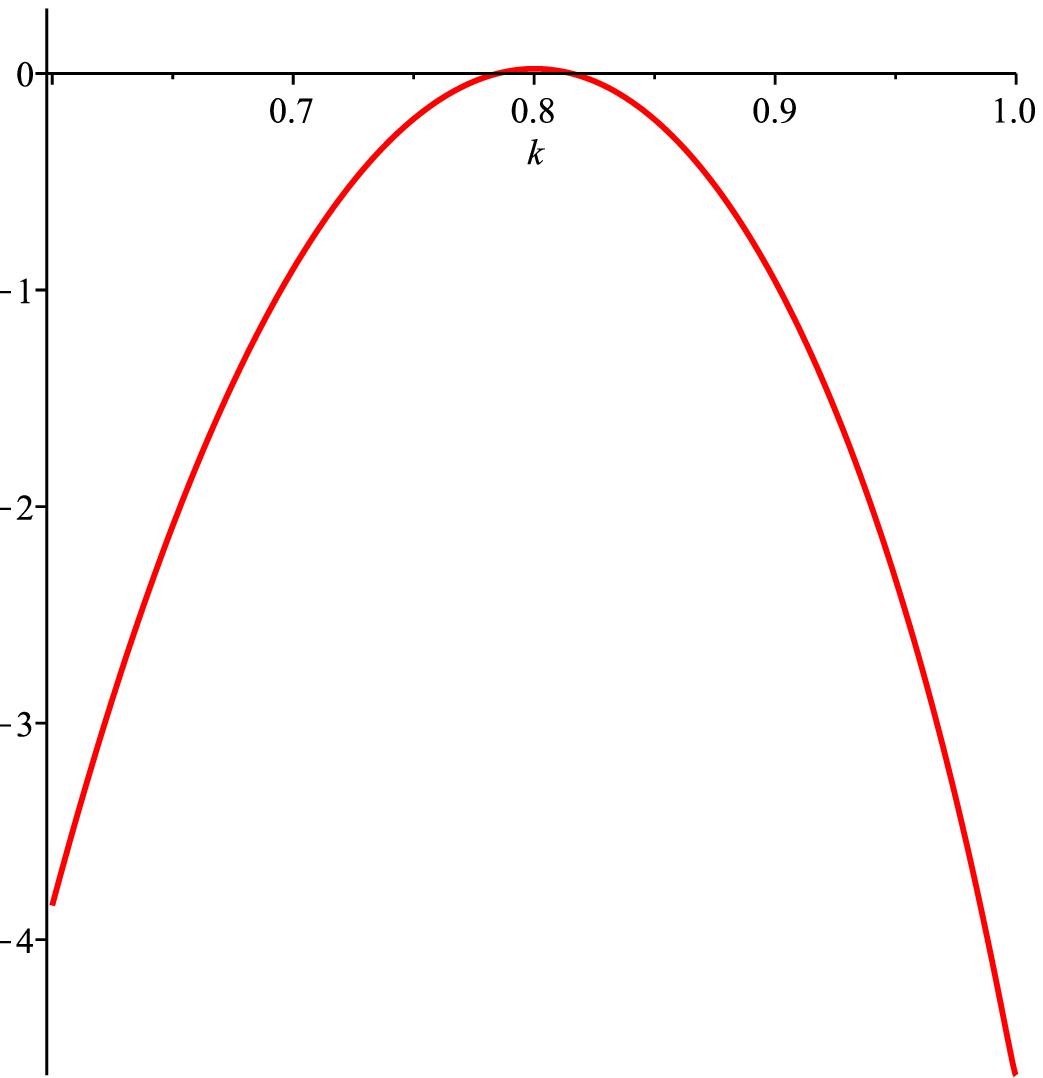}}
	\quad
	\subfigure[$L=400$]{
		\includegraphics[width=4.95cm,height=3.6cm]{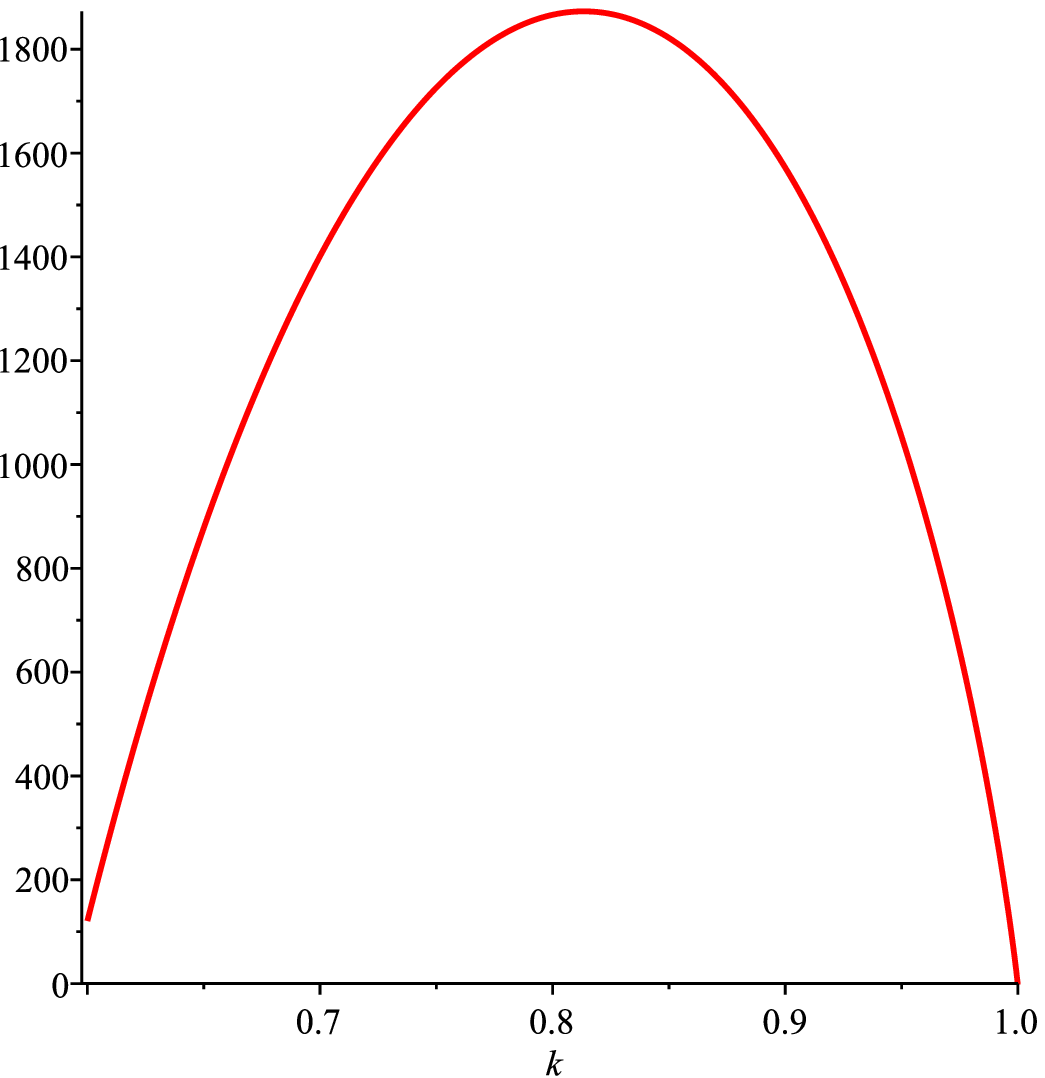}}
	\caption{Graphic of $d''$ for some values of $L > 2\pi$ in the interval $(k^*,1)$.}
	\label{figura2}
\end{figure}

\subsection{Orbital Instability for the cnoidal wave}\label{SubsectionInstability}

Let $L > 2\pi$ be fixed. As above, let us consider $\varphi = \varphi_c$ the cnoidal periodic solution given in  Theorem \ref{cnoidalcurve}.  Using the arguments in Subsection \ref{subsectionconvexityKG} together with Theorem \ref{eveneigenfunctions}, we obtain that the difference ${\rm n}(\mathcal{L}_e) - {\rm p}(d''(c))= 3 - 0 = 3$ is an odd number for all $c\in (-1,-c^*]\cup [c^*,1)$. For the case $L\in (2\pi,L_0)$, we also obtain that the difference is three for all $c\in(-c^*,c^*)$. This fact enables us to conclude the linear instability of periodic waves according to \cite[Theorem 5.1]{grillakis2} and our intention is to determine the orbital (nonlinear) instability. To do so, we use the main result in \cite{ShatahStraussbook} given by:

\begin{theorem} \label{sstehorem}
	In a Banach space $Y$, consider the evolution equation
	\begin{equation}\label{lineqA}
		\frac{du}{dt}=\mathcal{A}u+\mathcal{K}(u),
	\end{equation}
	where $\mathcal{A}$ is a linear operator and $\mathcal{K}$ is a nonlinearity. Assume the
	following:
	\begin{itemize}
		\item[{\bf (a)}] $\mathcal{A}$ generates a strongly continuous semigroup on $Y$.
		\item[{\bf (b)}] The spectrum of $\mathcal{A}$ meets the right half-plane
		$\{\lambda\in\mathbb{C};\ Re(\lambda)>0\}$.
		\item[{\bf (c)}] $\mathcal{K}(0)=0$, $\mathcal{K}:Y\rightarrow Y$ is continuous, and there exist
		$\rho_0>0$ and $\alpha_0>0$ such that $\|\mathcal{K}(u)\|_Y\leq
		\alpha_0\|u\|^{2}$ for all $\|u\|_Y<\rho_0$.
	\end{itemize}
	Then the zero solution is nonlinearly unstable.
\end{theorem}
\begin{proof}
	See Theorem 1 in \cite{ShatahStraussbook}.
\end{proof}

By Theorem $\ref{sstehorem}$, we obtain the following orbital instability result for the Klein-Gordon equation.

\begin{theorem}\label{instabilityKlein-Gordon}
Let $L>2\pi$ be fixed. For $c\in(-1,1)$, consider $\varphi$ the cnoidal periodic solution given by Theorem \ref{cnoidalcurve}.\\
i) If $c \in(-1,-c^*]\cup [c^*,1)$, where $c^*\in (0,1)$ is given by \eqref{valuew*}, then the periodic wave $\Phi$ is orbitally unstable in $X_e$ in the sense of Definition $\ref{stadef}$. \\
ii) For $L\in (2\pi,L_0)$, where $L_0\approx 20.354$, the periodic wave $\Phi$ is orbitally unstable for all $c\in (-1,1)$.
\end{theorem}
\begin{proof}
Our intention is to apply Theorem $\ref{sstehorem}$. In fact, let us consider the following perturbation of the wave $(\varphi,c\varphi,0,0)$ over the restricted space $Y=X_e$ given by
\begin{equation}\label{pertubation}
V(t)=T_1(-ct)U(t)-\Phi,
\end{equation}
where $\Phi=(\varphi, c \varphi,0,0)$, $c\in(-1,1)$ and $t\geq0$. Since $J^{-1}=-J$, we can see by the equations $(\ref{KF2})$ and $(\ref{KF3})$ that  
\begin{eqnarray}\label{pertubation2}
\frac{dV}{dt}&=& J \mathcal{L}V+J^{-1}(E'(\Phi+V)-E'(\Phi)-E''(\Phi)V) \nonumber \\
& = &  J \mathcal{L}V+J(E'(\Phi)-E'(\Phi+V)+E''(\Phi)V) \nonumber \\
&=&  J\mathcal{L}V+JN(V)=:\mathcal{A}V+\mathcal{K}V,
\end{eqnarray}
where
\begin{equation}\label{explicityN}
N(V)=\begin{pmatrix}
|\varphi+u|^4 \varphi + |\varphi+u|^4\text{Re} \, u -|\varphi|^4\varphi-5|\varphi|^4\text{Re}\,u \\ 
0 \\ 
|\varphi+u|^4 \text{Im}\, u-|\varphi|^4\text{Im}\,u \\ 
0
\end{pmatrix}
\end{equation}
and $V=(\text{Re}\,u,\text{Im}\,v, \text{Im}\,u, \text{Re}\,v)$.

The procedure is similar as determined in \cite[Theorem 3.2]{NataliPastor2014} and it suffices to show that \eqref{pertubation2} has the zero as a nonlinearly unstable solution. We prove that items {\bf (a)}-{\bf (c)} in Theorem $\ref{sstehorem}$ occur in our context. 

\noindent {\bf (a)} We prove that $\mathcal{A}=J\mathcal{L}$ is an infinitesimal generator of a $C_0$-semigroup on $X_e$. Indeed, observe that
$$
\mathcal{A}=\begin{pmatrix}
0 & 0 & 0 & 1 \\ 
0 & 0 & -(-\partial_x^2+1) & 0 \\ 
0 & 1 & 0 & 0 \\ 
-(-\partial_x^2+1) & 0 & 0 & 0
\end{pmatrix} 
+
\begin{pmatrix}
0 & 0 & c & 0 \\ 
0 & 0 & \varphi^4 & -c \\ 
-c & 0 & 0 & 0 \\ 
5\varphi^4 & c & 0 & 0
\end{pmatrix} 
 =: A+P.
$$
Since $P: X_e \longrightarrow X_e$ is a bounded operator, we see by \cite[Chapter 3, Theorem 1.1]{pazy} that it suffices to show that $A$ is infinitesimal generator of a $C_0$-semigroup on $X_e$. In fact, it is easy to see that the operator $-\partial_x^2+1$ is a positive and self-adjoint operator, so that $B=(-\partial_x^2+1)^{\frac{1}{2}}$ is well defined. We can define the inner product, for all $f=(u_1,u_2,v_1,v_2),g=(\tilde{u}_1,\tilde{u}_2,\tilde{v}_1,\tilde{v}_2) \in X_e$ as 
$$
\big((u_1,u_2,v_1,v_2), (\tilde{u}_1,\tilde{u}_2,\tilde{v}_1,\tilde{v}_2)\big)_{X_e}:=(Bu_1,B\tilde{u}_1)_{L^2}+(u_2,\tilde{u}_2)_{L^2}+(Bv_1,B\tilde{v}_1)_{L^2}+(v_2,\tilde{v}_2)_{L^2}
$$
which is equivalent to the usual inner product in $X_e$. From this and the self-adjointness of $B$, it follows that $A$ is a skew-adjoint operator. The remainder of the result follows by Stone's Theorem.

\vspace{10pt}

\noindent {\bf (b)} We claim that the spectrum of $\mathcal{A}$ intersects the set $\{\lambda \in \mathbb{C}\; ; \; \text{Re}(\lambda)>0\}$. In fact, by the arguments in Subsection \ref{subsectionconvexityKG} together with Theorem \ref{eveneigenfunctions}, we obtain that the difference ${\rm n}(\mathcal{L}_e) - {\rm p}(d''(c))= 3 - 0 = 3$ is an odd number for all $c\in (-1,-c^*]\cup [c^*,1)$. For the case $L\in (2\pi,L_0)$, we also obtain that the difference is three for all $c\in(-c^*,c^*)$. Thus, by \cite[Theorem 5.1]{grillakis2} our claim is now established.

\vspace{10pt}

\noindent {\bf (c)} The operator $\mathcal{K}=JN: X_e \longrightarrow X_e$ is well defined and satisfies $\mathcal{K}(0)=0$. Moreover, by the explicit representation of $N$ in \eqref{explicityN}, there exists $\alpha_0>0$ such that for every $V \in X_e$ satisfying $\|V\|_{X_e}\leq 1$ we have, 
$$
\|\mathcal{K}(V)\|_{X_e} \leq \alpha_0 \|V\|_{X_e}^2.
$$

By {\bf (a)}-{\bf (c)} and Theorem $\ref{sstehorem}$, we conclude that the zero solution of \eqref{pertubation} is non-linearly unstable. A rudimentary calculation implies that $\Phi$ is also orbitally unstable and the theorem is now proved.
\end{proof}

\section{Orbital instability of the cnoidal wave solutions for the QNLS}\label{section6}

Let $L>0$ be fixed. The goal of this section is to establish a result of orbital instability 
for the periodic standing wave solutions of the form
$
u(x,t)= e^{i\omega t} \varphi_{\omega}(x),
$
where $ \omega \in \big(\tfrac{4\pi^2}{L^2}, \infty\big)$ and $\varphi$ is given in Theorem \ref{cnoidalcurveNLS}. We consider the complex energy space $Z:= \mathbb{H}_{per}^1$ and it is well known that the equation \eqref{NLS-equation} also is invariant under  two basic symmetries: translation and rotation. The energy space in this case is $Z_e:= \mathbb{H}_{per,e}^1$ 
and as in the Definition \ref{stadef} we have the following definition.

\begin{definition}\label{stadefevenNLS}
We say that the periodic standing wave solution $u(x,t)=e^{i\omega t}\varphi(x)$ is orbitally stable  if for all $\varepsilon>0$ there exists $\delta>0$ such that for $u_0 \in Z_e$ such that $\|u_0-\varphi\|_{Z}<\delta, $ then the solution $u$ of \eqref{NLS-equation}, with initial data $u(0)=u_0$ can be extended globally to $t\geq0$  and
$$
\sup_{t\geq0}\inf_{\theta \in \mathbb{R}}\big\|u(t)-e^{i\theta}\varphi \big\|_Z< \varepsilon.
$$
Otherwise, the standing wave solution is said  orbitally unstable.
\end{definition}

Before establishing our stability result, we present the following well-posedness result associated to the QNLS equation.

\begin{theorem}\label{wellposedness1NLS}
The Cauchy problem associated with \eqref{NLS-equation} is locally well-posed in $Z_e$. More precisely, for each $u_0 \in Z_e$ there exists $T>0$ and a unique solution $u \in C([0,T], Z_e)$ of the equation \eqref{NLS-equation} with
$u(0)=u_0$. In addition, for all $T_0 \in (0,T)$ the mapping
\begin{equation*}
	u_0 \in Z_e \longrightarrow u \in C([0,T_0], Z_e)
\end{equation*}
is continuous.
\end{theorem}
\begin{proof}
See \cite[Theorem 2.3]{bourgain} (see also \cite[Corollary 5.6]{Iorio}).
\end{proof}

\subsection{Convexity of the function $\boldsymbol{\mathsf{d}}$.} \label{subsectionconvexityNLS}

 Let $L>0$ be fixed. For $\omega \in \mathcal{I}$, consider $\varphi$ the cnoidal wave given by Theorem $\ref{cnoidalcurveNLS}$. Since $\omega \in \mathcal{I} \longmapsto\varphi_{\omega}$ is smooth, we can define $\mathsf{d}: \mathcal{I} \subset \mathbb{R}_+ \longrightarrow \mathbb{R} $ as
$
\mathsf{d}(\omega)=\mathcal{E}(\varphi,0)+\omega \mathcal{F}(\varphi,0).
$
Since $\mathcal{G}'(\varphi,0)=0$, we have
\begin{equation}\label{5.0NLS}
\mathsf{d}''(\omega) = \frac{d}{d\omega} \mathcal{F}(\varphi,0)=\frac{1}{2}\frac{d}{d\omega}\int_0^L (\varphi(x))^2\; dx.
\end{equation}

By \eqref{Heumann}, \eqref{5.0NLS}  and the chain rule, we can write
\begin{eqnarray}\label{d''NLS}
\mathsf{d}''(\omega)=\frac{1}{2}\frac{d}{dk}\left(\int_0^L (\varphi(x))^2\; dx \right) \left( \frac{d\omega}{dk}\right)^{-1}=\frac{1}{2}\;\mathrm{R}'(k) \left( \frac{d\omega}{dk}\right)^{-1}.
\end{eqnarray}

Since $\mathrm{R}'(k)<0$ for all $k\in(0,k^*)$, we have by  \eqref{positivitydw} and \eqref{d''NLS} that
\begin{equation}\label{d2}
\mathsf{d}''(\omega)<0, \; \text{for all} \; \omega \in (\omega^*, \infty)
\end{equation}
where
\begin{equation}\label{valuew*2}
\omega^* \simeq \frac{42.83}{L^2}>0
\end{equation}
is defined uniquely by  \eqref{valuew} and  $k^* \simeq 0.593$.

Due to $(\ref{d2})$, we obtain that ${\rm p}(\mathsf{d}''(\omega))=0$ for all $\omega \in (\omega^*,+\infty)$.

\subsection{Orbital Instability for the cnoidal wave}

By Theorem \ref{eveneigenfunctionsNLS}, we have 
$$
\text{n}(\mathcal{L}_e) = 3 \quad \text{and} \quad \Ker(\mathcal{L}_e)=[(0,\varphi)].
$$

In addition, by Subsection \ref{subsectionconvexityNLS} and the Theorem \ref{eveneigenfunctionsNLS} we obtain that the difference \linebreak ${\rm n}(\mathcal{L}_e) - {\rm p}(\mathsf{d}''(\omega))= 3 - 0 = 3$ is an odd number for all $\omega\in (\omega^*,+\infty)$. Thus, we have the following result of orbital instability.

\begin{theorem}
Let $L>0$ be fixed. For $\omega \in (\omega^*,\infty)$, where $\omega^*>0$ is given in \eqref{valuew*2} consider $\varphi$ as the cnoidal periodic solution given by  Theorem \ref{cnoidalcurveNLS}. The standing wave $u(x,t)=e^{i\omega t}\varphi(x)$ is orbitally unstable in $Z_e$ in the sense of Definition $\ref{stadefevenNLS}$. 
\end{theorem}
\begin{proof}
The proof of this result can be done using similar arguments of Theorem \ref{instabilityKlein-Gordon}. In fact, standing wave solution as in \eqref{NLS-1} can be written as 
$$
u(x,t)= T_2(\omega t) (\varphi,0)=
\begin{pmatrix}
\varphi(x) \cos(\omega t) \\ 
\varphi(x) \sin(\omega t)
\end{pmatrix}, 
$$
where 
$$
T_2(\theta)U=\begin{pmatrix}
\cos \theta & \sin \theta \\ 
-\sin \theta & \cos \theta
\end{pmatrix} \begin{pmatrix}
u_1 \\ 
u_2
\end{pmatrix},
$$
is the rotation symmetry and $U=(u_1,u_2) \in Z_e$.
\end{proof}

 \section*{Acknowledgments}
The authors would like to express their gratitude to F\'abio Natali for the suggestions and comments concerning this work.
This study is supported by CAPES/Brazil - Finance 001.

\end{document}